\documentclass[a4paper,10pt]{amsart}
\usepackage[utf8]{inputenc}
\usepackage{lmodern}
\usepackage[T1]{fontenc}
\usepackage{microtype} 	
\usepackage{amsmath,amssymb,amsfonts,amsthm,latexsym,mathrsfs}
\usepackage{stmaryrd} 
\usepackage{color}
\usepackage{hyperref}
\hypersetup{colorlinks=true,linkcolor=red,citecolor=blue}
\usepackage{multirow}
\usepackage{booktabs}
\usepackage{longtable}
\theoremstyle{plain}
\newtheorem{theorem}[subsection]{{\bf Theorem}}
\newtheorem*{theorem*}{{\bf Theorem}}
\newtheorem{corollary}[subsection]{{\bf Corollary}}
\newtheorem*{corollary*}{{\bf Corollary}}

\newtheorem{lemma}[subsection]{{\bf Lemma}}

\theoremstyle{definition}

\theoremstyle{remark}

\newtheorem{example}[subsection]{{\it Example}}
\numberwithin{equation}{section}

\DeclareMathOperator{\HH}{H}
\DeclareMathOperator{\B}{B}
\DeclareMathOperator{\M}{M}
\DeclareMathOperator{\dd}{d}

\DeclareMathOperator{\rank}{rank}
\DeclareMathOperator{\K}{K}

\DeclareMathOperator{\CC}{\mathbb C}

\begin{document}
\baselineskip=15pt
\title{Bogomolov multipliers of $p$-groups of maximal class}
\author[G. A. Fern\'{a}ndez-Alcober]{Gustavo A. Fern\'{a}ndez-Alcober}
\address[Gustavo A. Fern\'{a}ndez-Alcober]{
Department of Mathematics \\
University of the Basque Country UPV/EHU \\
48080 Bilbao \\
Spain}
\thanks{}
\email{gustavo.fernandez@ehu.eus}
\author[U. Jezernik]{Urban Jezernik}
\address[Urban Jezernik]{
Department of Mathematics \\
University of the Basque Country UPV/EHU \\
48080 Bilbao \\
Spain}
\thanks{This project was supported by the Spanish Government,
grant MTM2017-86802-P, partly with FEDER funds,
and by the Basque Government, grant IT974-16. 
This project has also received funding from the European Union’s Horizon 2020 research and innovation programme under the Marie Sklodowska-Curie grant agreement No 748129.}
\email{urban.jezernik@ehu.eus}
\keywords{Bogomolov multiplier, $p$-group of maximal class}
\date{\today}
\begin{abstract}
\noindent
Let $G$ be a $p$-group of maximal class and order $p^n$. We determine whether or
not the Bogomolov multiplier $\B_0(G)$ is trivial in terms of the lower central
series of $G$ and $P_1 = C_G(\gamma_2(G) / \gamma_4(G))$.  If in addition $G$ has positive degree of commutativity and
$P_1$ is metabelian, we show how understanding $\B_0(G)$ reduces to the
simpler commutator structure of $P_1$. This result covers all $p$-groups of
maximal class of large enough order and, furthermore, it allows us to give the
first natural family of $p$-groups containing an abundance of groups with
nontrivial Bogomolov multipliers.
We also provide more general results on Bogomolov multipliers of
$p$-groups of arbitrary coclass $r$.
\end{abstract}
\maketitle
\section{Introduction}
\label{s:intro}

\noindent
The Bogomolov multiplier is a group-theoretical invariant, originally introduced 
as an obstruction to the famous {\em rationality problem} in
algebraic geometry. This problem asks whether a given field extension $E/k$ is
rational (purely transcendental). Of particular interest is the situation when a
finite group $G$ acts on the function field $L$ of an
affine space by permuting indeterminates. The subfield $L^G$ of fixed points of
this action represents the function field of the quotient variety. 
Noether \cite{Noe13} asked whether the extension $L^G/k$ is always (stably) purely
transcendental. This version of the rationality problem is known as {\em
Noether's problem}. 
Saltman \cite{Sal84} found explicit examples of groups for which the answer to
Noether's problem is in general  negative, even when taking $k = \CC$. His
approach was to use a
certain Galois-cohomological invariant associated to the group $G$, namely, the
{\em unramified Brauer group}. It was Bogomolov \cite{Bog87} who showed
that this invariant can be computed purely in terms of group-cohomology
as the intersection of kernels of all restriction maps from the
Schur multiplier of $G$ to Schur multipliers of its abelian subgroups:
\[
\bigcap_{A \leq G, [A,A] = 1} \ker \left( \HH^2(G, \mathbb{C}^*) \longrightarrow \HH^2(A, \mathbb{C}^*) \right).
\]
This cohomological invariant is now known as the
{\em Bogomolov multiplier} and denoted by $\B_0(G)$.
It represents an obstruction to Noether's problem, which has a negative
answer for $G$ provided that $\B_0(G)$ is nontrivial.

Explicitly determining the Bogomolov multiplier of a group $G$ can be quite involved.
By general homological arguments, it boils down to studying 
Bogomolov multipliers of the Sylow subgroups of $G$. 
For $p$ a prime, the smallest examples of $p$-groups with nontrivial Bogomolov multiplier
are of order $64$ if $p=2$ \cite{Chu08,Chu10}, and 
of order $p^5$ if $p$ is odd \cite{Bog87,Mor12p5}. These groups are all of maximum nilpotency class
given their order.

Motivated by these examples, in this paper we set to inspect Bogomolov multipliers of
$p$-groups $G$ of {\em maximal class}, that is, of order $p^n$ and nilpotency class $n-1$.
Our study of Bogomolov multipliers of these groups partially builds on and continues previous works on Schur multipliers 
of $p$-groups of maximal class and related groups \cite{DEF08,Eic08,EF11,Mor11}.
The theory of $p$-groups of maximal class was started by Blackburn in
\cite{Bla58} and further developed by Leedham-Green and McKay \cite{Lee76,Lee78a,Lee78b}, by
Shepherd \cite{She70}, and by the first author \cite{Fer95}.
General accounts of this theory can be found in \cite{Fer01}, \cite{Hup67} or \cite{Lee02}.
We can assume $n \geq 4$, and we will tacitly do so in what follows.
Consider the chief series
$G>P_1>\cdots>P_n=1$, with $P_i=\gamma_i(G)$ for $i\geq 2$, and $P_1=C_G(P_2/P_4)$.
Then $P_1$ is a maximal subgroup of $G$.
As a main result, we prove that most $p$-groups of maximal class have non-trivial Bogomolov multipliers.

\begin{theorem}
\label{th:UCR:MaxClassNontriviality}
Let $G$ be a $p$-group of maximal class and order $p^n$.
Then $\B_0(G)$ is trivial if and only if $[P_1, P_1] = [P_1, P_{n-2}]$.
\end{theorem}

The {\em degree of commutativity\/} of $G$ is the largest integer $\ell\le n-3$
such that $[P_i,P_j]\leq P_{i+j+\ell}$ for all $i,j\ge 1$. If we assume in the above theorem that $G$ has {\em positive degree of commutativity}, i.e., $\ell > 0$, then we have that $G$ has trivial Bogomolov multiplier precisely when $P_1$ is abelian. This assumption on $\ell$ is very natural when dealing with $p$-groups of maximal class, and in fact it is satisfied whenever $n > p + 1$ \cite[Theorem 3.3.5]{Lee02}.

\begin{corollary}
Let $G$ be a $p$-group of maximal class of positive degree of commutativity.
Then $B_0(G)$ is trivial if and only if $P_1$ is abelian.
\end{corollary}

We remark that an abelian-by-cyclic finite group is already known to have trivial
Bogomolov multiplier by \cite[Lemma 4.9]{Bog87}.
As a simple example, consider the special cases $p \in \{2,3\}$.
When $p=2$, all the groups have abelian $P_1$, and therefore trivial Bogomolov multipliers.
When $p=3$, all the groups have positive degree of commutativity and $|[P_1,
P_1]| \leq 3$, so we either have that $P_1$ is abelian, in which case $\B_0(G)$
is trivial, or $[P_1, P_1] = P_{n-1}$ and the degree is positive, in which case
$\B_0(G)$ is nontrivial. 

Explicitly determining the structure of $\B_0(G)$ is rather complicated
even in the case when $P_1$ is assumed to be metabelian, still assuming that $G$
has positive degree of commutativity. 
As we will see, both assumptions are satisfied if $n$ is large enough with respect to $p$.
Therefore we focus on this case,
and we give in Theorem \ref{th:UCR:MaxClassP1} a precise description of the Bogomolov multiplier in terms
of the curly exterior square introduced by Moravec \cite{Mor11}.

Remarkably, this allows us to describe the first natural example of a family of $p$-groups having Bogomolov multipliers of unbounded exponent (see Section \ref{s:examples}).

\begin{theorem}
\label{th:MaxClassConstructionLarge}
Let $p\ge 5$ be a prime, and let $m$ and $n$ be integers such that $m\ge 4$ and $m\le n\le 2m-2$.
If we write $n-m+1 = x(p-1)+y$ with $x\ge 0$ and $0\le y<p-1$, then there exists a $p$-group of maximal class of order $p^n$ with $P_1'=P_m$ and Bogomolov multiplier isomorphic to
\[
C_{p^{x+1}}^{\lfloor y/2 \rfloor} \times C_{p^{x}}^{\lfloor (p-1-y)/2 \rfloor}.
\]
\end{theorem}

Finally we also explore Bogomolov multipliers of $p$-groups of general coclass
$r$, that is, groups of order $p^n$ and nilpotency class $n-r$ for some $r \geq
1$. In taking $r=1$, one recovers finite $p$-groups of maximal class.
By general coclass theory (see Section \ref{s:coclass}), there is a way of
collecting all but finitely many of these groups into a finite number of
infinite trees arising from pro-$p$ groups. We investigate the structure of
Bogomolov as well as Schur multipliers from this point of view; we improve
a result on the asymptotic shape of the
Schur multiplier for infinite coclass sequences of groups
\cite[Theorem 1]{EF11}, and we extend the same result to the Bogomolov
multiplier.

\section{Maximal class}
\label{s:bogz}

\noindent
In this section, we prove the general criterion for deciding whether or not Bogomolov multipliers of $p$-groups of maximal class are trivial,
Theorem \ref{th:UCR:MaxClassNontriviality}.
This is done by first proving the more technical Theorem \ref{th:UCR:MaxClassP1},
from which we can deduce one of the implications of
Theorem \ref{th:UCR:MaxClassNontriviality}, and then we complete the proof of this theorem.

Before proceeding, we will collect some well-known structural properties of $p$-groups of maximal class that we will require later \cite{Fer01,Hup67,Lee02}.
First of all, recall that we have a chief series $G>P_1>\cdots>P_n=1$
with $P_i = \gamma_i(G)$ for $i \geq 2$. Observe that $P_i=1$ for $i\ge n$, and that $P_2,\ldots,P_n$ are the
only normal subgroups of $G$ of index greater than $p$.
If we pick arbitrary elements $s\in G\setminus (P_1\cup C_G(P_{n-2}))$ and $s_1\in P_1\setminus P_2$, then
$s$ and $s_1$ generate $G$.
We call $s$ a uniformizing element.
Also, we have $C_G(s)=\langle s \rangle P_{n-1}$, which is of order $p^2$.
Let us define $s_i=[s_{i-1},s]$ for all $i\ge 2$.
Then $s_i\in P_i\setminus P_{i+1}$ for $1\le i\le n-1$, and $s_i=1$ for $i\ge n$.
As a consequence, every $g\in G$ can be uniquely written in the form
\begin{equation*}
g = s^{i_0}s_1^{i_1}\ldots s_{n-1}^{i_{n-1}},
\end{equation*}
with $0\le i_j<p$ for all $j=0,\ldots,n-1$.
We refer to this as the {\em normal form\/} of $g$ with respect to $s$ and the $s_i$.
Let $\ell$ be the degree of commutativity of $G$.
Then $\ell\ge 0$, and $\ell=n-3$ if and only if $P_1$ is abelian.
We write $\ell(G)$ when we want to emphasize the group $G$.
A fundamental result of Blackburn is that $\ell>0$ if $n\ge p+2$.
On the other hand, we have $\ell(G/N)\ge \ell(G)$ for every $N\lhd G$.
Also, $\ell(G/Z(G))$ is always positive and, as a consequence, $\ell(G)=0$ if and only if $[P_1,P_{n-2}]=P_{n-1}$.

Let us also recall some of the basic facts related to Bogomolov multipliers
$\B_0(G)$ for arbitrary groups $G$.
In this paper, we will work exclusively with a combinatorial interpretation
of the Bogomolov multiplier as developed by Moravec \cite{Mor12}.
Namely, consider the {\em curly exterior square} $G\curlywedge G$,
which is the group generated by the symbols $x\curlywedge y$ for all pairs $x,y\in G$, subject to the
relations
\[	
\begin{aligned}
\label{eq:ext}
xy\curlywedge z = (x^y\curlywedge z^y)(y\curlywedge z), \quad 
x\curlywedge yz = (x\curlywedge z)(x^z\curlywedge y^z), \quad
a\curlywedge b = 1,
\end{aligned}
\]
for all $x,y,z\in G$ and all $a,b\in G$ with $[a,b]=1$.
By \cite[Theorem 3.2]{Mor12}, there is a canonical epimorphism $G\curlywedge G\to [G,G]$, whose kernel is isomorphic
to $\B_0(G)$.
In this sense, the Bogomolov multiplier can be interpreted as a measure of how the commutator relations in a group fail to follow from the so-called universal ones and from commuting
pairs in $G$, see \cite{Jez13cp}.
This can be taken a step further if one has an explicit presentation of
the underlying group $G$ as $F/R$ with $F$ a free group.
Let $\K(F) = \{ [f,g] \mid f, g \in F \}$ denote the set of
simple commutators in $F$. 
Note that $\langle \K(F) \cap R \rangle$ contains $[R,F]$.
We then have Hopf-type formulas \cite[Proposition 3.8]{Mor12} for both the curly exterior square and the Bogomolov multiplier,
\[
G \curlywedge G \cong \frac{F'}{\langle \K(F) \cap R \rangle}
\quad \textrm{and} \quad
\B_0(G) \cong \frac{F' \cap R}{\langle \K(F) \cap R \rangle}.
\]
Our approach to proving statements about a $p$-group of maximal class $G$ will be to use the above formulae. We also remark that in such a group, 
all elements of $G'$ are simple commutators of the form $[s,g]$ with $g\in G$
and then, by \cite[Theorem 3.11]{Mor12}, 
we have an epimorphism $\B_0(G)\longrightarrow \B_0(G/P_i)$ for every $i\ge 2$.

We are now ready to give a precise description of the Bogomolov multiplier 
of most $p$-groups of maximal class in terms of the curly exterior square
of a maximal subgroup.

\begin{theorem}
\label{th:UCR:MaxClassP1}
Let $G$ be a $p$-group of maximal class with positive degree of commutativity
and $P_1'$ abelian. Let $s \in G \setminus P_1$ be a uniformizing element.
Then $\B_0(G)$ is isomorphic to the coinvariants  $(P_1 \curlywedge P_1)_{\langle s \rangle}$.
\end{theorem}

Here, $P_1\curlywedge P_1$ is seen as an $\langle s \rangle$-module under the action induced by the rule
$(x\curlywedge y)^s=x^s\curlywedge y^s$ for all $x,y\in P_1$.

\begin{proof}[Proof of Theorem \ref{th:UCR:MaxClassP1}]
It is known that any $p$-group of maximal class $G$ admits
a presentation that can be constructed as follows (see \cite[Exercise 3.3(4)]{Lee02}).
Let $F$ be the free group on $t$ and $t_1$. Denote $t_i =
[t_{i-1},t]$ for all $i\ge 2$. Every element $g$ of $G$ has a normal form in terms
of the generating set $s$ and $s_i$ for $1 \leq i \leq n-1$. For a word $w$ of
$F$, let $\llbracket w \rrbracket$ denote the word in $t$ and $t_i$ for $1 \leq
i \leq n-1$ obtained by replacing $s$ with $t$ and $s_i$ with $t_i$ in the
normal form of the element of $G$ that is represented by the word $w$. Denote by
$\rho(w) = w \llbracket w \rrbracket^{-1}$ the relator associated to $w$.
Set
\[
\mathcal R_0 = \{ t_n \}, \;
\mathcal R_1 = \{ \rho(t^p), \rho((tt_1)^p) \}, \;
\mathcal R_2 = \{ \rho([s_{2i}, s_1]) \mid 1 \leq i \leq (p-1)/2 \},
\] 
and let $R$ be the normal subgroup of $F$ generated by $\mathcal R_0 \cup \mathcal R_1 \cup \mathcal R_2$.
Then $F/R$ is a presentation of the group $G$.
We denote by $M = \langle t_1 \rangle^F R$ the maximal subgroup of $F$ that maps onto $P_1$, and we consider the map
\[
\lambda \colon M' \longrightarrow \frac{F' \cap R}{\langle \K(F) \cap R \rangle}
\cong \B_0(G), \qquad
w \longmapsto \rho(w) \langle \K(F) \cap R \rangle.
\]
Observe that $(P_1 \curlywedge P_1)_{\langle s \rangle}$ is isomorphic
to the group $M'/\left( \langle \K(M) \cap R \rangle [M', t] \right)$.
Our objective is to prove that $\lambda$ induces an isomorphism
\[
\bar \lambda \colon \frac{M'}{\langle \K(M) \cap R \rangle [M', t]} \cong (P_1 \curlywedge P_1)_{\langle s \rangle} \longrightarrow \frac{F' \cap R}{\langle \K(F) \cap R \rangle}
\cong \B_0(G).
\]

{\bf I.} {\em We show that $\langle t_m, \dots, t_{n-1} \rangle \cap R \leq \langle \K(F) \cap R \rangle$},
where $m$ is the smallest index with $[P_1, P_1] = P_m$. 
To this end, let $\omega \in \langle t_m, \dots, t_{n-1} \rangle \cap R$.
Observe that since $P_m$ is assumed to be abelian, we have
$[\gamma_m(F), \gamma_m(F)] \leq \langle \K(F) \cap R \rangle$, and so
$\omega$ can be written as
\[
\omega \equiv \textstyle \prod_{i=m}^{n-1} t_i^{a_i} \pmod{\langle \K(F) \cap R \rangle}
\]
with $a_m \equiv 0 \pmod{p}$. Note that
\[
\textstyle [t_{m-1}^p, t] = t_m^{t_{m-1}^{p-1}} t_m^{t_{m-1}^{p-2}} \cdots t_m \equiv t_m^p [t_m, t_{m-1}^{\binom{p}{2}}] \pmod{\langle \K(F) \cap R \rangle}.
\]
Now, if $p > 2$, then $t_{m-1}^{\binom{p}{2}} \in P_m$, and in
the case when $p = 2$, $P_{m-1} = P_{n-1}$ is central in $G$.
We therefore always have
$[t_m, t_{m-1}^{\binom{p}{2}}] \in \K(F) \cap R$.
It follows that 
\[
\textstyle \omega \equiv \left[ \prod_{i=m}^{n-1} t_{i-1}^{a_i}, t \right]
\pmod{\langle \K(F) \cap R \rangle},
\]
and so $\omega \in \langle \K(F) \cap R \rangle$, as desired.

{\bf II.} {\em We claim that $\lambda$
is a homomorphism.} To see this, first observe that since
$P_m$ is assumed to be abelian, we have
$[M' \gamma_m(F),M' \gamma_m(F)] \leq R$.
Now pick any $x,y \in M'$. Note that $\llbracket x \rrbracket, \llbracket y \rrbracket \in \gamma_m(F)$. Hence
\[
\lambda(x)\lambda(y) =
x \llbracket x \rrbracket^{-1} y \llbracket y \rrbracket^{-1} \equiv
xy (\llbracket x \rrbracket \llbracket y \rrbracket)^{-1}
\pmod{\langle \K(F) \cap R \rangle}.
\]
We have $\llbracket x \rrbracket \llbracket y \rrbracket \llbracket xy \rrbracket^{-1} \in R \cap \langle t_m, \dots, t_{n-1} \rangle \leq \langle \K(F) \cap R \rangle$.
We thus conclude
\[
\lambda(x) \lambda(y) = x y \llbracket xy \rrbracket^{-1} = \lambda(xy)
\pmod{\langle \K(F) \cap R \rangle}.
\]

{\bf III.} {\em Let us now show that $\lambda$ is surjective.} Consider the group $F/\langle R \cap \K(F) \rangle$. Its subgroup $R/\langle R
\cap \K(F) \rangle$ is an abelian group that can be generated by the cosets of
the elements of $\mathcal R_1 \cup \mathcal R_2$. Observe that $R/(R \cap F')
\cong RF'/F'$ can be generated by the cosets of elements of $\mathcal
R_1$. Moreover, the elements $t^pF'$ and $t_1^pF'$ form a base of the free
abelian group $RF'/F'$ of rank $2$. Hence we have that the torsion group
$(R \cap F')/\langle R \cap \K(F) \rangle$ is generated by the cosets of the
elements of $\mathcal R_2$. Now note that $\mathcal R_2 \subseteq \rho(M')$. Therefore $\lambda$ is indeed surjective.

{\bf IV.} {\em The homomorphism $\lambda$ factors through $\langle \K(M) \cap R \rangle [M', t]$.}
It is clear that $\langle \K(M) \cap R \rangle$ is contained in the kernel of $\lambda$. To see that the same holds for $[M',t]$, consider
$\lambda(m^t)$ for some $m \in M'$. We have that
$\llbracket m^t \rrbracket \equiv \llbracket m \rrbracket^t$ modulo 
$( R \cap \langle t_m, \dots, t_{n-1} \rangle ) \langle t_n \rangle^F \leq \langle \K(F) \cap R \rangle$.
Whence
\[
\lambda(m^t) = m^t \llbracket m^t \rrbracket^{-1} \equiv m^t \llbracket m \rrbracket^{-t} = m \llbracket m \rrbracket^{-1} [m \llbracket m \rrbracket^{-1},t] \equiv \lambda(m) \pmod{\langle \K(F) \cap R \rangle},
\]
proving our claim. There is thus an induced homomorphism
\[
\bar \lambda \colon \frac{M'}{\langle \K(M) \cap R \rangle [M', t]} \cong (P_1 \curlywedge P_1)_{\langle s \rangle} \longrightarrow \frac{F' \cap R}{\langle \K(F) \cap R \rangle}
\cong \B_0(G).
\]

{\bf V.} {\em The homomorphism $\bar \lambda$ is injective.}
To this end, let $w \in M'$ represent an element in $\ker \bar \lambda$.
So $w \llbracket w \rrbracket^{-1} \in \langle \K(F) \cap R \rangle$.
Write $w = \llbracket w \rrbracket \cdot \prod_i [x_i, y_i]$ for some $[x_i, y_i] \in R$.
Collect all those indices $i$ for which $x_i, y_i \in M$ into a set $I$. Eventually replacing $w$ by $w \prod_{i \in I} [x_i, y_i]^{-1}$, we may assume that
$x_i \not\in M$ and $y_i \in M$ for all indices $i$. Write $x_i = t^{a_i} m_i$ for
some $1 \leq a_i < p$ and $m_i \in M$.
Since $[x_i, y_i] \in R$, 
it follows that $y_i = t_{n-1}^{b_i} r_i$ for some $0 \leq b_i < p$ and $r_i \in R$.
So $[x_i, y_i] = [t^{a_i} m_i, t_{n-1}^{b_i} r_i] \equiv [t^{a_i}, t_{n-1}^{b_i} r_i] [m_i, t_{n-1}^{b_i} r_i]$.
Now $[m_i, t_{n-1}^{b_i} r_i] \in \langle \K(M) \cap R \rangle$, so that
$[x_i, y_i] \equiv [t, t_{n-1}]^{a_i b_i} [t, r_i]^{a_i}$.
Thus we may write
$w \equiv \llbracket w \rrbracket \cdot [t, t_{n-1}^a] [t, r]$ for some
$0 \leq a < p$ and $r \in R$.
As $R \leq \langle M', t^p, t_m, \dots, t_{n-1} \rangle^F$, it follows that
$R$ can be generated modulo $R \cap M'$ by words in $t^p, t_m, \dots, t_{n-1}$.
We are in a setting where $[M',t] \equiv 1$, so 
we can write
$w \equiv \prod_{i \geq m} t_i^{c_i}$ for some integers $c_i$.

Note that since $P_m$ is assumed to be abelian,
the image of the group $\langle t_i \mid i \geq m \rangle$
in $M/\langle \K(M) \cap R \rangle$ is abelian, and so it is a quotient of
the free abelian group generated by the elements $t_i$ for $i \geq m$.
Moreover, the group $M/M'$ is the quotient of the free abelian group
generated by the elements $t^p$ and $t_i$
for $i \geq 1$ subject only to the relations
\[
\prod_{i = 1}^p t_{j + i}^{\binom{p}{i}} \equiv 1 \pmod{M'}
\]
for all $j \geq 1$. These arise from expanding
$[t_j, t^p] \in M'$.
The element $w$ belongs to $M'$,
therefore it can be written modulo $\langle \K(M) \cap R \rangle$
as a product of some powers of elements of the form
$\prod_{i=1}^p t_{j+i}^{\binom{p}{i}}$ for some $j \geq m$.
Now, observe that
\[
\prod_{i=1}^p t_{j+i}^{\binom{p}{i}} \equiv [t_j, t^p] \equiv 1 \pmod{\langle \K(M) \cap R \rangle}
\]
for all $j \geq m-1$.
Therefore $w \equiv 1$ in the domain of $\bar\lambda$ and the proof is complete.
\end{proof}

In the next corollary, we use arithmetical conditions to have high degree of commutativity
from \cite{Fer95} in order to show that Theorem \ref{th:UCR:MaxClassP1} covers all but finitely many $p$-groups of maximal class for every prime $p$.

\begin{corollary}
Let $G$ be a $p$-group of maximal class of order $p^n$.
If
\[
n\ge \max\{p+2,6p-29\},
\]
then $\B_0(G) \cong (P_1 \curlywedge P_1)_{\langle s \rangle}$.
\end{corollary}

\begin{proof}
The condition $n\ge p+2$ ensures that $G$ has positive degree of commutativity.
Hence, according to Theorem \ref{th:UCR:MaxClassP1} it suffices to check that $P_1$ is metabelian.
If $p=2$, $3$ or $5$, then the degree of commutativity $\ell$ of $G$ is at least $n-3$, $n-4$ or $(n-6)/2$, respectively, and the result readily follows.
On the other hand, if $p\ge 7$ then $2\ell \ge n-2p+5$ by \cite{Fer95}.
Since $P_1'=[P_1,P_2]\le P_{\ell+3}$, we have
\[
[P_1',P_1'] \le [P_{\ell+3},P_{\ell+3}] = [P_{\ell+3},P_{\ell+4}] \le P_{3\ell+7},
\]
where
\[
3\ell + 7 \ge \frac{3}{2}(n-2p+5) + 7 = n + \frac{1}{2}(n-6p+29) \ge n,
\]
since $n\ge 6p-29$ by hypothesis.
Thus $P_1$ is metabelian, as desired.
\end{proof}

Now we continue by proving Theorem \ref{th:UCR:MaxClassNontriviality}.
We will use Theorem \ref{th:UCR:MaxClassP1} for this together with
the fact that, if $[P_1, P_1] = P_m$, then $\B_0(P)$ maps onto $\B_0(P / P_{m+1})$.

\begin{proof}[Proof of Theorem \ref{th:UCR:MaxClassNontriviality}]
We first assume that $[P_1,P_1]=P_m$ is strictly larger than $[P_1,P_{n-2}]$, and prove that $\B_0(G)$ is nontrivial.
If $\ell(G)=0$ then $[P_1,P_{n-2}]=P_{n-1}$, and so $G/P_{m+1}$ is a proper quotient of $G$.
We conclude that $\ell(G/P_{m+1})>0$ in every case.
Then we may apply Theorem \ref{th:UCR:MaxClassP1} to $G/P_{m+1}$ to get
$\B_0(G/P_{m+1})\cong (P_1/P_{m+1} \curlywedge P_1/P_{m+1})_G$.
Since $\B_0(G)$ surjects onto $\B_0(G/P_{m+1})$ and $(P_1/P_{m+1} \curlywedge P_1/P_{m+1})_G$ surjects onto
$[P_1/P_{m+1},P_1/P_{m+1}]_G=(P_m/P_{m+1})_G\cong C_p$, we conclude that $\B_0(G)\ne 1$.

Now we prove the converse, namely that the condition $[P_1,P_1]=[P_1,P_{n-2}]$ implies that $\B_0(G)=1$.
If $P_1$ is abelian, then $G$ is abelian-by-cyclic and hence $\B_0(G)$ is trivial by \cite{Bog87}.
So assume that $P_1$ is not abelian.
The restriction $[P_1, P_1] = [P_1, P_{n-2}]$ gives $[P_1, P_1] = P_{n-1}$.
Note that $P_{n-1}$ is generated by the element $s_{n-1}$ of order $p$.
Moreover, since $[P_1, P_{n-2}] = P_{n-1}$ and $[P_2, P_{n-2}] = 1$, there exists a $\lambda \neq 0$ mod $p$ with $[s_1, s_{n-2}] = s_{n-1}^\lambda$.
The latter equality may be rewritten as $[s^{\lambda} s_1, s_{n-2}] = 1$.
Expanding $1 = s^\lambda s_1 \curlywedge s_{n-2}$ in the curly exterior square $G \curlywedge G$ gives
\[
	(s^\lambda \curlywedge s_{n-2})^{s_1} (s_1 \curlywedge s_{n-2}) =
	(s^\lambda [s^\lambda, s_1] \curlywedge s_{n-2} s_{n-1}^{-\lambda}) (s_1 \curlywedge s_{n-2}) =
	(s \curlywedge s_{n-2})^\lambda (s_1 \curlywedge s_{n-2}),
\]
therefore $(s_{n-2} \curlywedge s)^\lambda =  (s_1 \curlywedge s_{n-2})$.
Furthermore, pick any $s_i, s_j, s_k, s_l$ in $P_1$ and assume that both of the elementary wedges $s_i \curlywedge s_j$ and $s_k \curlywedge s_l$ are nontrivial in $G \curlywedge G$.
As $[P_1, P_1] = P_{n-1}$, both of the commutators $[s_i, s_j]$ and $[s_k, s_l]$ equal a power of $s_{n-1}$.
Since $\B_0(P_1)$ is trivial by \cite[Corollary 4.1]{Jez13cp}, there exists an $m > 0$ such that $(s_i \curlywedge s_j) (s_k \curlywedge s_l)^m \in \B_0(P_1)$ is trivial.
The natural homomorphism $\B_0(P_1) \to \B_0(G)$ shows that $(s_i \curlywedge s_j) (s_k \curlywedge s_l)^m$ is also trivial in $G \curlywedge G$.
Hence all the elementary wedges $s_i \curlywedge s_j$ are equal to a power of the nontrivial one $s_1 \curlywedge s_{n-2}$. 

Now let $w$ be an arbitrary element of $\B_0(G)$.
For any $x,y \in P_1$ and $g,h \in G$, we have $[x \curlywedge y, g \curlywedge h] = [x,y] \curlywedge [g,h] = 1$ in $G \curlywedge G$, since $P_{n-1} = Z(G)$.
Note also that $[s_i \curlywedge s, s_j \curlywedge s] = s_{i+1} \curlywedge s_{j+1}$.
Therefore the element $w$ can be written as
\[
	\textstyle w = \prod_{i=1}^{n-2} (s_i \curlywedge s)^{\alpha_i} \cdot (s_1 \curlywedge s_{n-2})^{\beta}
\]
for some integers $\alpha_i, \beta$.
Observe that
\[
1 = s_i \curlywedge s^p = (s_i \curlywedge s^{p-1}) (s_i \curlywedge s)^{s^{p-1}}
=  \textstyle   (s_i \curlywedge s)^p \cdot \prod_{j > i} (s_j \curlywedge s)^{a_j} \cdot (s_1 \curlywedge s_{n-2})^{b}
\]
for some $a_j,b$.
We may thus assume that $0 \leq \alpha_i < p$, and the same for $\beta$.
Note that $w$ belongs to $\B_0(G)$ if and only if we have $\prod_{i=1}^{n-2} s_{i+1}^{\alpha_i} \cdot [s_1, s_{n-2}]^{\beta} = 1$ in $G$.
Collecting the left hand side in its normal form and comparing exponents gives $\alpha_i = 0$ for all $i \leq n-3$ and $\alpha_{n-2} + \lambda \beta = 0$.
We thus have $w = ((s \curlywedge s_{n-2})^{\lambda} (s_1 \curlywedge s_{n-2}))^{\beta}$, and so $w = 1$ by above.
Hence $\B_0(G)$ is trivial, as required.
\end{proof}

%
%
%

\section{Examples}
\label{s:examples}

\noindent
In this section,  we show how Theorem \ref{th:UCR:MaxClassP1} can be utilized to
explicitly determine the structure of Bogomolov multipliers of some particular
$p$-groups of maximal class (for $p\ge 5$). As a consequence, we prove Theorem
\ref{th:MaxClassConstructionLarge}.

To do this, we recall that the structure of a $p$-group $G$ of maximal class with $P_1$ of nilpotency class $2$ can be given in terms of the ring of integers in the $p$-th cyclotomic number field $\mathcal O$.
So $\mathcal O = {\mathbb Z}[\theta]/(1+ \theta + \cdots + \theta^{p-1})$, where $\theta$
is a primitive complex $p$-th root of unity. Denote $\kappa = \theta - 1$ and let $\mathfrak p = (\kappa)$.
There is an action of $\mathcal O$ on $P_m$ with $\theta$ acting via conjugation by $s$.
By \cite[Lemma 8.2.1]{Lee02}, there is an $\mathcal O$-module isomorphism
between $P_i/P_{i+j}$ and $\mathcal O / \mathfrak p^j$, induced by the map
\[
\mathcal O \to P_i/P_{i+j}, \quad \sum_u a_u \kappa^u \mapsto \prod_u s_{i+u}^{a_u}.
\]
The commutator structure of $P_1$ can thus be understood in terms of the 
homomorphism
\[
\alpha \colon \mathcal O/\mathfrak p^{m-1} \wedge \mathcal O/\mathfrak p^{m-1}
\to
\mathcal O/\mathfrak p^{n-m} \cong P_m.
\]
This is in fact a homomorphism of $\langle \theta \rangle = C_p$ modules.
Set
\[
\K \alpha =
\langle \ker \alpha \cap \{ {\textrm{elementary wedges in }}\mathcal O/\mathfrak p^{m-1} \wedge \mathcal O/\mathfrak p^{m-1} \} \rangle.
\]
Now consider the induced epimorphism
\[
\alpha_{C_p} \colon
(\mathcal O/\mathfrak p^{m-1} \wedge \mathcal O/\mathfrak p^{m-1})_{C_p}
\to
(\mathcal O/\mathfrak p^{n-m})_{C_p} \cong \mathcal O/\mathfrak p \cong P_m/P_{m+1} \cong C_p
\]
obtained by factoring out the action of $\theta$. Correspondingly, there is the
induced kernel
\[
\K \alpha_{C_p} = \langle \ker \alpha_{C_p} \cap \{ {\textrm{image of elementary wedges in }}(\mathcal O/\mathfrak p^{m-1} \wedge \mathcal O/\mathfrak p^{m-1})_{C_p} \} \rangle.
\]
Notice that
\[
\frac{(\mathcal O/\mathfrak p^{m-1} \wedge \mathcal O/\mathfrak p^{m-1})_{C_p}}
{\K \alpha_{C_p}}
\cong
\left( \frac{\mathcal O/\mathfrak p^{m-1} \wedge \mathcal O/\mathfrak p^{m-1}}
{\K \alpha} \right)_{C_p}
\]
by right-exactness of coinvariants. We make the following identification:
\[
P_1 \curlywedge P_1 = \frac{P_1/P_m \wedge P_1/P_m}{\langle xP_m \wedge yP_m \mid [x,y] = 1 \rangle} \cong
\frac{\mathcal O/\mathfrak p^{m-1} \wedge \mathcal O/\mathfrak p^{m-1}}
{\K \alpha}
\]
Now, to provide concrete examples, we show that by carefully
selecting the map $\alpha$, which in turn determines the group $G$, one may achieve that the image of the map
$\K \alpha_{C_p}$
in $(\mathcal O/\mathfrak p^{m-1} \wedge \mathcal O/\mathfrak
p^{m-1})_{C_p}$ is trivial. Based on the above identification, this amounts to constructing groups $G$ with $\B_0(G) \cong
(\mathcal O/\mathfrak p^{m-1} \wedge \mathcal O/\mathfrak p^{m-1})_{C_p}$. Such
a commutator structure will therefore produce groups whose Bogomolov multipliers
will have largest possible rank and exponent for the given values of $n$ and
$m$. Furthermore, essentially the same argument will deal with quotients of such
extreme groups. The construction we give below covers this more general case.

Fix any $m \geq 4$ and set $\ell = m - 3$.  The number $\ell$ will be the degree of
commutativity of the constructed group. Now pick any $n$ satisfying $m < n \leq
2 m - 2$. Set ${\mu} = n - m + 2$, so that $2 < {\mu} \leq m$. Let $g$ be a
primitive root modulo $p$ and pick an integer $a$ so that $a \equiv (g + 1)^{-1}
\pmod{p}$. It is here that we need $p\ge 5$. In the case when $a > (p-1)/2$,
replace $a$ by $1-a$, so that in the end, $2 \leq a \leq (p-1)/2$. Now define $
\alpha \colon \mathcal O/\mathfrak p^{m-1} \wedge \mathcal O/\mathfrak p^{m-1}
\to \mathcal O/\mathfrak p^{n-m} $ by the rule
\[
\alpha(x \wedge y) = \kappa^{-1} \cdot(\sigma_a(x)\sigma_{1-a}(y) - \sigma_a(y)\sigma_{1-a}(x))
\]
for $x,y \in \mathcal O/\mathfrak p^{m-1}$. Here, $\sigma_a$ is the automorphism
of $\mathcal O/\mathfrak p^{m-1}$ which maps $\theta$ to $\theta^a$.
This corresponds to the map induced by $\kappa^{-1} S_a$ in \cite[Theorem 8.3.1]{Lee02}. 
Set $u_a = (\theta^a - 1)/\kappa \in \mathcal O^*$. Then
\[
\alpha(\kappa^i \wedge \kappa^j) = \mathrm{sgn}(i-j) \kappa^{i+j-1} (u_a u_{1-a})^{\min \{ i,j \}} (u_a^{|i-j|} - u_{1-a}^{|i-j|}) \in \mathfrak p^{i+j-1}.
\]
Observe that $u_a^{|i-j|} - u_{1-a}^{|i-j|} \equiv a^{|i-j|} - (1-a)^{|i-j|}
\pmod{\mathfrak p}$. This element belongs to $\mathfrak p$ if and only if we
have $(a^{-1} - 1)^{|i-j|} \equiv 1 \pmod{p}$. By our choice of $a$, this occurs
precisely when $i \equiv j \pmod{p-1}$. The commutator map $\alpha$ therefore
satisfies $\alpha(\kappa^i \wedge \kappa^j) \in \mathfrak p^{i+j-1} \backslash
\mathfrak p^{i+j}$ whenever $i \not\equiv j \pmod{p-1}$.

Invoking \cite[Theorem 8.2.7]{Lee02}, there is a $p$-group $G$ of maximal class
of order $p^n$ whose commutator structure is described by the map $\alpha$ given
above. In terms of the $P_i$-series of $G$, the above discussion shows that we
have $[P_{i}, P_{j}] = P_{i+j+\ell}$ for all $i,j \geq 1$ that satisfy $i
\not\equiv j \pmod{p-1}$.

This highly restricted commutator structure enables us to completely understand
commuting pairs of $G$.

\begin{lemma}
\label{lemma:comm-pairs}
Let $x \in P_i \backslash P_{i+1}$. Then
$C_{P_i}(x) = \langle x, P_{i+j} \rangle$, where $j = \max \{ n - 2i - \ell, 1 \}$.
\end{lemma}

\begin{proof}
Clearly the right hand side centralizes $x$. Conversely, suppose that $y \in P_k
\backslash P_{k+1}$ for some $k > i$ and $[x,y] = 1$. Assume that $y \notin
\langle x \rangle$. If $k \equiv i \pmod{p-1}$, then $y = x^{r} z$ for some $r >
0$ and $z \in P_{k'} \backslash P_{k'+1}$ with $[z,x] = 1$ and $k' \not\equiv i
\pmod{p-1}$. In this case, replace $y$ by $z$ and $k$ by $k'$, so that we may
assume $k \not\equiv i \pmod{p-1}$. Now, since $[P_i, P_k] = P_{i+k+\ell}$ and
$[P_{i+1},P_k][P_i,P_{k+1}] \leq P_{i+k+\ell+1}$, it follows that $P_{i+k+\ell} =
P_{i+k+\ell+1}$, which is only possible when $i+k+\ell \geq n$.
\end{proof}

In particular, note that $Z(P_1) \geq P_{{\mu}} \geq P_m$ in the group $G$.
Transferring to the $C_p$-module $\mathcal O/\mathfrak p^{m-1} \wedge \mathcal
O/\mathfrak p^{m-1}$, we thus have that the elementary wedges in $\mathfrak
p^{\mu - 1}/\mathfrak p^{m-1} \wedge \mathcal O/\mathfrak p^{m-1}$ are all
contained in $\K \alpha$. Using Lemma \ref{lemma:comm-pairs} more precisely, we
now show that wedges that arise from commuting pairs are, modulo the action of
$C_p$, nothing but the latter.

\begin{lemma} \label{lemma:k-alpha}
$\K \alpha_{C_p} = (\mathcal O/\mathfrak p^{m-1} \wedge \mathfrak p^{{\mu}-1}/\mathfrak p^{m-1}) + [\mathcal O/\mathfrak p^{m-1} \wedge \mathcal O/\mathfrak p^{m-1}, C_p]$.
\end{lemma}
\begin{proof}

Let $x \wedge y \in \K \alpha$ for some $x,y \in \mathcal O/\mathfrak p^{m-1}$.
Suppose that $x$ corresponds to an element in $P_i \backslash P_{i+1}$ and $y$
to an element in $P_j \backslash P_{j+1}$ with $i \leq j$. We will prove that $x
\wedge y$ is equivalent to an element of the submodule $\mathfrak p^{{\mu} -
1}/\mathfrak p^{m-1} \wedge \mathcal O/p^{m-1}$ modulo $[\mathcal O/\mathfrak
p^{m-1} \wedge \mathcal O/\mathfrak p^{m-1}, C_p]$ by induction on $i$.

If $i \equiv j \pmod{p-1}$, then as in the proof of Lemma 
\ref{lemma:comm-pairs},
we may write $y = x^r z$ with $z \in P_{j'}\backslash P_{j'+1}$ and $j'
\not\equiv i \pmod{p-1}$. Then $x \wedge y = x \wedge z$, so we may without loss
of generality assume that $i \not\equiv j \pmod{p-1}$. By the lemma, we then
have $i + j + \ell \geq n$. If $i = 1$, this implies that $j \geq n - \ell - 1 =
{\mu}$, whence $x \wedge y \in \mathcal O/\mathfrak p^{m-1} \wedge \mathfrak
p^{{\mu}-1}/\mathfrak p^{m-1}$. This is the base for the induction.
Suppose now that $i > 1$.  Then $x = \kappa \tilde x$ for some $\tilde x \in
\mathcal O/\mathfrak p^{m-1}$ corresponding to a group element in $P_{i-1}
\backslash P_{i}$. Observe that
\begin{equation*}
\label{eq:trick}
\tilde x \wedge \kappa y + x \wedge y + x \wedge \kappa y = \kappa(\tilde x \wedge y) \in [\mathcal O/\mathfrak p^{m-1} \wedge \mathcal O/\mathfrak p^{m-1}, C_p].
\end{equation*}
Note that $\kappa y$ corresponds to a group element in $P_{j+1}$, and therefore
$\tilde x \wedge \kappa y$ and $x \wedge \kappa y$ both belong to $\K \alpha$.
Using this reasoning, we show our claim by reverse induction on $j$. When $j
\geq {\mu}$, it is clear that $x \wedge y \in \mathcal O/\mathfrak p^{m-1}
\wedge \mathfrak p^{{\mu}-1}/\mathfrak p^{m-1}$. Assume now that $j < {\mu}$. By
induction, both $\tilde x \wedge \kappa y$ (since $\tilde x$ belong to a higher
level) and $x \wedge \kappa y$ (since $\kappa y$ belongs to a lower level) are
contained in $\mathcal O/\mathfrak p^{m-1} \wedge \mathfrak
p^{{\mu}-1}/\mathfrak p^{m-1}$ modulo $[\mathcal O/\mathfrak p^{m-1} \wedge
\mathcal O/\mathfrak p^{m-1}, C_p]$. An application of \eqref{eq:trick} then
implies that the same holds for $x \wedge y$, as claimed.
\end{proof}

The above gives that
\[
\B_0(G)
=
\frac{(\mathcal O/\mathfrak p^{m-1} \wedge \mathcal O/\mathfrak p^{m-1})_{C_p}}{\K \alpha_{C_P}}
=
(\mathcal O /\mathfrak p^{{\mu} -1} \wedge \mathcal O /\mathfrak p^{{\mu} -1})_{C_p}.
\]

Finally, a structure description of the group $(\mathcal O/\mathfrak p^{{\mu}-1}
\wedge \mathcal O/\mathfrak p^{{\mu}-1})_{C_p}$ may be read off from the explicit
$C_p$-module decomposition of $\mathcal O/\mathfrak p^{{\mu}-1} \wedge \mathcal
O/\mathfrak p^{{\mu}-1}$ into a direct sum of cyclic submodules as given in
\cite[Theorem 8.13]{Lee78b}.
Then Theorem \ref{th:MaxClassConstructionLarge} readily follows.

Let us consider some special cases of Theorem \ref{th:MaxClassConstructionLarge}.
When $n$ is chosen so that $n \equiv m-1
\pmod{p-1}$, we obtain a group $G$ with $\B_0(G)$ homocyclic of rank $(p-1)/2$
and exponent $p^{(n-m+1)/(p-1)}$. Further selecting $n \approx 2m$, we have the
property $\exp \B_0(G) \approx \sqrt{\exp G}$. Consider now the option $n =
m+1$. In this case, we obtain groups that are immediate descendants of groups on
the main line of the maximal class tree. Their Bogomolov multipliers are $C_p$. In the
very special case when $m = 4$, we obtain the known groups of order $p^5$ with nontrivial
Bogomolov multipliers.  Another extreme option is picking $n =
2m - 2$. In this case, we have $n - m + 1 = m - 1$, so by varying $m$, the
groups exhaust all the possibilities for the Bogomolov multiplier, depending on
the value $m-1 \pmod{p-1}$. Finally, consider the option of selecting
consecutive values $n = m+1, m+2, \dots, 2m - 2$. In terms of the constructed
groups, this corresponds to a path in the maximal class tree, starting from an
immediate descendant of a group on the main line (which is of order $p^m$) and
going deeper into the branch. In this process, the value $n - m + 1$ grows one
by one, so that the corresponding Bogomolov multipliers grow in size by $p$ on
each second step, starting with $C_p$ for the group closest to the main line.
The growth is ``staircase''-like, consecutively increasing the orders of the
generators by a factor of $p$ on each second step.

The above value is in fact the largest possible for the rank of the Bogomolov multiplier of any $p$-group of maximal class, as the following result shows.

\begin{corollary}
\label{c:rank-bound-max}
Let $G$ be a $p$-group of maximal class. 
Then $\rank \B_0(G) \leq (p-1)/2$.
\end{corollary}
\begin{proof}
In analogy to the Schur multiplier, given a presentation of any finite group
$G$ consisting of $d$ generators and $r$ relations among which $r_k$ are commutators,
the rank of $\B_0(G)$ is bounded by $r - r_k - d$ (see \cite[Corollary 5.2]{JM16}).
Now, we recalled at the beginning of the proof of Theorem \ref{th:UCR:MaxClassP1}
that any $p$-group of maximal class admits a presentation with $d = 2$,
$r = (p+5)/2$ and $r_k = 1$, and this proves the result.
\end{proof}

The above exponents are also the largest possible for groups in which $P_1$ is of nilpotency class $2$. This is so because if $[P_1, P_1] = P_m$, then $\B_0(G)$ is a quotient of the
group $(\mathcal O/\mathfrak p^{m-1} \wedge \mathcal O/\mathfrak p^{m-1})_{C_p}$,
whose exponent is
$p^{\lceil (n - m + 1)/(p-1) \rceil} \leq p^{\lceil n/2(p-1) \rceil} \approx \sqrt{\exp G}$.

\section{Higher coclass}
\label{s:coclass}

In this final section, we investigate Bogomolov and Schur multipliers of finite
$p$-groups of a fixed coclass $r$. Referring to \cite{Lee02,EF11}
we recall a few fundamental facts about coclass theory. Groups of coclass
$r$ can be collected into a graph $\Gamma(p,r)$ whose vertices correspond
to their isomorphism types and there is an edge between $G$ and $H$ if and only
if $G \cong H/N$ where $N$ is the last term of the lower central series of $H$.
An infinite pro-$p$ group $S$ of coclass $r$ corresponds precisely to 
an infinite path in the above graph, and there are only finitely many of them.
One can find a minimal positive integer $i$, the so-called primary root (see \cite[Section 2.1]{EF11}), with the property that the coclass of
$S/\gamma_i(S)$ is precisely $r$ and that the quotient $S/\gamma_i(S)$ is not
isomorphic to any $\bar S / \gamma_i(\bar S)$  for any other pro-$p$ group $\bar S$ 
of coclass $r$. We denote by 
$\mathcal{T}(S)$ the subtree of $\Gamma(p,r)$ rooted at $S/\gamma_i(S)$.
The tree $\mathcal{T}(S)$ has a main line given by the lower central quotients
of $S$ and it is in general of unbounded width. 

As shown by the examples in the
previous section, it is possible to find groups with Bogomolov multipliers of
arbitrary exponents as one moves further away from the main line. Via the theory
of CP covers \cite{JM16}, one can modify these examples to produce groups of arbitrary 
coclass with the same property. 
These groups further possess the property that they are stem, i.e.,
their center is contained in their derived subgroup.

\begin{theorem}
Let $p \geq 5$ be a prime and $r$ a positive integer. 
For every $x \geq r \geq 1$ there exist stem $p$-groups of coclass $r$
with Bogomolov multiplier of exponent $p^x$.
\end{theorem}
\begin{proof}
Let $H$ be a $p$-group of maximal class with the
properties that $\exp \B_0(H) = p^x$ and that there exists $L \leq \B_0(H)$ 
with $\exp L = p^x$ and $|\B_0(H):L| = p^{r-1}$.
Such a group $H$ exists by Theorem \ref{th:MaxClassConstructionLarge}.
Let $G$ be a CP covering group (see \cite[Section 4]{JM16}) of $H$,
and let $Q = G / L$.
By \cite[Lemma 4.6]{JM16},
the group $G$ is a CP covering group of $Q$ with kernel $L$,
and we have $\B_0(Q) \cong L$.
Therefore $|Q| = |G| / |L| = |H| \cdot |\B_0(H)| / |L| = |H| \cdot p^{r-1}$.
At the same time, 
it follows from \cite[Lemma 4.8]{JM16} that $Q/Z(Q) \cong H/Z(H)$,
so $Q$ has the same nilpotency class as $H$. Therefore $Q$ is of coclass $r$. 
Now, as $H$ is a group of maximal class, its center is equal to the last
term of its lower central series, so $H$ is a stem group. The center
$Z(Q)$ is generated by $\B_0(H)/L$ together with the preimage of $Z(H)$
under the projection $Q \to H$.
Since we also have $\B_0(H)/L \leq [Q,Q]$,
it follows that $Q$ is a stem group.
\end{proof}

In the remainder, we will focus on bounding the sizes of Bogomolov multipliers
of groups of fixed coclass. First of all, we recall that the rank of
Schur multipliers (and therefore also Bogomolov multipliers) can be absolutely bounded.

\begin{theorem}[{\cite[Theorem 1]{AE13}}]
There is a constant $C = C(p,r)$ such that
for every finite $p$-group $G$ of coclass $r$, we have
$\dd(\B_0(G)) \leq \dd(\M(G)) \leq C$.
\end{theorem} 

Bounding the exponent of both multipliers is a more delicate issue. 
If we only consider finite $p$-groups of fixed coclass that are quotients
of an infinite pro-$p$ group, then the exponents of the Schur multipliers are
in general unbounded (see \cite[Theorem A]{Eic08}).
The exponent can nevertheless be bounded in terms of the exponent of the
underlying group. This is in accordance with a classical problem asking what the
relationship is between the exponent of a finite group and of its Schur
multiplier. More precisely, it is conjectured that $\exp \M(G)$ divides $\exp G$
for a finite $p$-group $G$ if $p > 2$. Confer \cite{Sam17} and the references therein.
In the context of $p$-groups of fixed coclass, it is known that one at least has  
a linear bound on the exponent of the Schur multiplier.

\begin{theorem}[{\cite[Theorem 3.3]{Sam17}}] \label{th:schur_exp}
There is a constant $D = D(p,r)$ such that
for every finite $p$-group $G$ of coclass $r$, we have
$\exp \M(G) \leq D \cdot \exp G$.
\end{theorem}

For any group of maximal class, the exponent of the Schur multiplier is bounded
by the exponent of the group (see \cite[Theorem 1.4]{Mor11}).

As for the Bogomolov multiplier, we now show that there is an
absolute bound for the exponent of Bogomolov multipliers of quotients of
infinite pro-$p$ groups that are on the main line of the coclass tree.

\begin{theorem} \label{th:bog_exp_main_line}
There is a constant $E = E(p,r)$ such that
for every infinite pro-$p$-group $S$ of coclass $r$, we have
$\exp \B_0(S/\gamma_i(S)) \leq E$ for all $i \geq 1$.
\end{theorem} 
\begin{proof}
By \cite[Theorem 10.1]{DDMS91}, we have that the group $S$
contains an abelian subgroup $A$ of index bounded in terms of $p$
and $r$. 
It now follows from \cite[Proposition 6.2]{JM16} that
$\exp \B_0(S/\gamma_i(S)) \leq |S:A\gamma_i(S)| \cdot \exp \B_0(A \gamma_i(S)/\gamma_i(S)) = |S:A\gamma_i(S)| \leq |S:A|$.
\end{proof}

The groups in Theorem \ref{th:MaxClassConstructionLarge} show how 
groups far away from the main line of the coclass tree $\Gamma(p,1)$ can have Bogomolov
multipliers whose exponents grow without limits.  More can be said about the
asymptotic structure of both multipliers of groups that are boundedly away from
the main line in $\mathcal{T}(S)$ for any pro-$p$ group $S$ of coclass $r$. More precisely, one can consider the shaved
subtree $\mathcal{T}(S,k)$ consisting of those groups in $\mathcal{T}(S)$ that
are of distance at most $k$ from the main line. Its branches
$\mathcal{B}_j(S,k)$ consist of descendants of $S/\gamma_j(S)$ that are not
descendants of $S/\gamma_{j+1}(S)$. It is proved in \cite{EL08} that these
branches satisfy a periodic pattern. This periodicity can be made explicit on
the level of groups; each group in the first branch determines an infinite
coclass sequence $(G_x)_{x \in \mathbb{N}}$. 
For each infinite sequence, 
the groups are uniformly described by a single parameterised presentation
(see \cite{EF11} as well as \cite[Chapter 8]{Dor10} for more details),
and so their Schur multipliers relate to the parameter $x$ and can also be 
uniformly described. 
It is known that sizes of these Schur multipliers will grow
as the orders of the underlying groups grow (see \cite[Theorem A]{Eic08}).
Based on the above given bounds for the exponent, the uniform
description of Schur multipliers from \cite{EF11} can be refined and extended
to the Bogomolov multiplier as follows.

\begin{theorem}
Let $(G_x)_{x \in \mathbb{N}}$ be an infinite coclass sequence.
Then there exist a finite abelian $p$-group $A$ and
integers $s_1, \dots, s_m$ such that
\[
\M(G_x) \cong A \times \prod_{i=1}^m C_{p^{x + s_i}}
\]
for almost all $x$, and for $p > 2$ one has $m > 0$.
Moreover, the groups $\B_0(G_x)$ are almost all pairwise isomorphic.
\end{theorem}
\begin{proof}
By \cite[Theorem 1]{EF11} there exist integers $r_1, \dots, r_m$ and 
$s_1, \dots, s_m$ so that the Schur multiplier $\M(G_x)$ is of the form
\begin{equation} \label{eq:schur_multiplier_form_in_family}
\prod_{i=1}^m C_{p^{r_i x + s_i}}
\end{equation}
for almost all $x$. 
At the same time, it follows from \cite[Section 9]{EL08} 
(see also \cite[Remark 5.23]{Dor10}) that 
each group $G_x$ fits into an exact sequence
\begin{equation} \label{eq:ses_gx}
1 \longrightarrow
T_x \longrightarrow
G_x \longrightarrow
R \longrightarrow
1, 
\end{equation}
where $R$ is a fixed finite $p$-group and $T_x$ is a homocyclic abelian group (in fact, an $R$-module) of
exponent $p^{x + e}$ for some fixed $e$. 
This gives that $\exp G_x \leq \exp R \cdot \exp T_x = \exp R \cdot p^{x + e}$.
Now, it follows from Theorem \ref{th:schur_exp} that
$\exp \M(G_x)$ is linearly bounded in terms of $\exp G_x$,
which implies that all $r_1, \dots, r_m$ must be either $0$ or $1$.
Together with \cite[Theorem A]{Eic08}, this proves the first claim.

We now turn to Bogomolov multipliers.
Suppose $R$ has a polycyclic presentation with $d$ generators $g_1, \dots, g_d$,
and let $T = \mathbb{Z}_p^n$, so that $T_x = T / p^{x+e} T$.
The group $G_x$ has a central extension $G_x^*$,
constructed by adding a new central generator $y_{i,j}$ for $1 \leq j \leq i \leq d + n$ 
as a tail of every one relation in the polycyclic presentation for $G_x$
(see \cite[Section 8.1]{Dor10}).
Suppose there are $\ell$ such relations.
Setting $Y_x^* = \langle y_{i,j} \mid 1 \leq j \leq i \leq d + n \rangle$,
one has $\M(G_x) \cong \mathrm{Torsion}(Y_x^*)$. 
The presentation of $G_x^*$ might not be consistent (see \cite[Section 8.3]{Dor10}),
so one needs to impose relations between the central generators $y_{i,j}$
arising from consistency relations. These relations are all of the form
\begin{equation} \label{eq:relations_tails_consistency}
\prod_{1 \leq j \leq i \leq d + n} y_{i,j}^{\theta_{i,j}(x)} = 1
\end{equation}
for some $\theta_{i,j}(x) \in \mathbb{Z}_p[p^x]$, and can therefore
be collected into a consistency matrix $A(x)$ with coefficients in 
$\mathbb{Z}_p[p^x]$ (see \cite[Definition 8.11]{Dor10}).
After imposing the consistency relations, one sees that
$\mathrm{Torsion}(Y_x^*)$ is of the form \eqref{eq:schur_multiplier_form_in_family}
(see \cite[Theorem 8.14]{Dor10}).
Now, it follows from \cite[Proposition 2.1]{JM14} that in order to obtain
$\B_0(G_x)$, we need to add extra relations to the consistency matrix $A(x)$
in order to obtain a matrix $B(x)$.
These relations arise in the following way. For any pair of commuting
elements $z_1, z_2$ in $G_x$, we pick their lifts $z_1^*, z_2^*$ 
in the group $G_x^*$ and add the relation $[z_1^*, z_2^*] \in Y_x^*$
to the matrix $B(x)$. If $Z_x^*$ is the group obtained by factoring $\mathbb{Z}_p^\ell$
by the relations imposed by $B(x)$, then we have $\B_0(G_x) \cong \mathrm{Torsion}(Z_x^*)$. By applying to the same reasoning as in the case of the Schur multiplier, it therefore
suffices to show that the relations arising from commuting
pairs in $G_x$ are of the form \eqref{eq:relations_tails_consistency}. This can be seen as follows. Fix sections of the projections $G_x \to R$ and $T \to T_x$.
Every commuting pair in $G_x$ is of the form $w_1 t_1$, $w_2 t_2$
for $w_1, w_2$ a lift of a commuting pair in $R$ and $t_1, t_2$ elements of
$T_x$. 

{\em Suppose first that $w_1 = w_2 = 1$.} This means that we are dealing with
commuting pairs in $T_x$. For each pair of generators of $T_x$, we evaluate
their commutator in $G_x^*$ and obtain $[t_1, t_2] \in Y_x^*$, which is simply
one of the generators $y_{i,j}$ and therefore of the form \eqref{eq:relations_tails_consistency}. 

{\em Next, consider the case when $w_2 = 1$.}
This means we are dealing with commuting pairs of the form $w_1 t_1, t_2$.
Since we already imposed the relation that $t_1, t_2$ commute in $Z_x^*$,
it suffices to deal with the case when the commuting pair is of the form
$w_1, t_2$. When $w_1$ is fixed, this amounts to specifying the kernel
of the associated $\mathbb{Z}_p$-linear map $T_x \to T_x$, $t \mapsto (1 - w_1) \cdot t$,
where $w_1 \in R$ acts on $T_x$ via \eqref{eq:ses_gx}. This action is in fact
descended from a fixed action of $R$ on $T$ (see \cite[Remark 5.23]{Dor10}).
Consider then the linear map $1 - w_1 \colon T \to T$ and put it into its 
Smith normal form $XDY$ with $D = \mathrm{diag}(p^{a_1}, \dots, p^{a_k}, 0, \dots, 0)$ for some non-negative integers $a_1 \geq \cdots \geq a_k$.
Set $v_i = Y^{-1} e_i$ for $1 \leq i \leq n$.
Then the kernel of the induced map $1 - w_1$ on $T_x$ is generated
by the vectors
\[
p^{x+e-a_1} v_1, \dots, p^{x+e-a_k} v_k, v_{k+1}, \dots, v_{n}.
\]
For each one of these generating vectors $v$, we evaluate its commutator
with $w_1$ in $G_x^*$ and obtain $[w_1, v] \in Y_x^*$, which is
(by \cite[Lemma 8.8]{Dor10}) again of the form \eqref{eq:relations_tails_consistency}.
An analogous argument works when we assume that $w_1 = 1$.

{\em Finally, consider the general case.} 
Let $w_1, w_2$ be a commuting pair in $R$.
In order to determine commuting pairs $w_1 t_1, w_2 t_2$ in $G_x$,
rewrite the commuting condition as
$w_1 t_1 w_2 t_2 = w_2 t_2 w_1 t_1$, which is equivalent to
$t_1^{w_2} t_2 = [w_2, w_1] t_2^{w_1} t_1$.
Writing $[w_2, w_1] = t^{\alpha(x)}$ for some $t \in T_x$ and
$\alpha(x) \in \mathbb{Z}_p[p^x]$
and interpreting the action of $R$ on $T_x$ as a linear map,
we can rewrite the commuting condition as
\begin{equation} \label{eq:linear_equation_ad_operators}
(1 - w_2) \cdot t_1
-
(1 - w_1) \cdot t_2
= t^{\alpha(x)}.
\end{equation}
This is a linear equation in $t_1, t_2$. As in the previous case,
one can see this as a linear transformation on $T$ and pass 
to a Smith normal form. We thus get a diagonal equation of the form
$D [\tilde t_1, \tilde t_2]^T = [\tilde t^{\tilde \alpha(x)}]$
for $\tilde t_1, \tilde t_2, \tilde t \in T$ and 
$\tilde \alpha(x) \in \mathbb{Z}_p[p^x]$.
The solutions of this equation, if they exist
(depending on the initial pair $w_1, w_2$), can be
obtained as sums of a particular solution $p_1, p_2$ and an element $k_1, k_2$ of the 
kernel, where both the particular solution and generators of the kernel are of the form 
$s^{\beta(x)}$ for some $s \in T$, $\beta(x) \in \mathbb{Z}_p[p^x]$ 
by the same argument as in the previous case.
Let us add to the matrix $B(x)$ the relations arising from
evaluating in $G_x^*$ the commutators 
$[w_1 p_1 b_1, w_2 p_2 b_2], [w_1 p_1, w_2 p_2] \in Y_x^*$
for $b_1, b_2$ basis vectors of the kernel corresponding to the linear equation 
\eqref{eq:linear_equation_ad_operators}.
By imposing these relations, we can verify that in the group $G_x^*$,
all commuting pairs of $G_x$ lift to commuting pairs. Indeed, for
$w_1, w_2 \in G_x$, particular solutions $p_1, p_2 \in T_x$,
and elements $k_1, k_2 \in T_x$ in the kernel, we can expand the commutator 
$[w_1 p_1 k_1, w_2 p_2 k_2]$ in $G_x^*$ modulo the relations collected in $B(x)$ into
\[
[w_1 p_1, w_2 p_2 k_2]^{k_1} [k_1, w_2 p_2 k_2] =
[w_1 p_1, k_2] [w_1 p_1, w_2 p_2] [k_1, w_2 p_2].
\]
Since $p_1, p_2$ is a particular solution, this
further simplifies to $[w_1, k_2] [k_1, w_2 ]$. 
The latter is linear in $k_1, k_2$, and these can be expanded by the basis of the 
kernel of \eqref{eq:linear_equation_ad_operators}.
As all of the commuting relations $[w_1 p_1 b_1, w_2 p_2 b_2]$
have been added to $B(x)$, it follows that $[w_1, k_2] [k_1, w_2 ] = 1$. 
We have therefore collected
all relations in $Y_x^*$ induced by commuting pairs in the matrix $B(x)$,
and they are all of the form \eqref{eq:relations_tails_consistency}, as required.

It follows that $\B_0(G_x)$ is of the form \eqref{eq:schur_multiplier_form_in_family}
for large enough values of $x$.
Now, recall that $G_x$ is boundedly far in the graph $\mathcal{T}(S)$ 
from a group on the mainline. This means that there is a surjection
$\pi \colon G_x \to S_j$ for a group $S_j = S / \gamma_j(S)$ on the main line
with $\ker\pi$ generated by commutators and $|\ker \pi| \leq p^C$ for some constant $C$.
Let $z \in \ker \pi \cap Z(G_x)$ be a central commutator of order $p$.
It follows from \cite[Proposition 4.1]{Mor11} 
(see also the proof of \cite[Theorem 7.6]{JM16}) that
\[
\ker \left( \B_0(G_x) \to \B_0(G_x/\langle z \rangle) \right) 
\leq
J_z = \langle x \curlywedge y \in G_x \curlywedge G_x \mid [x,y] = z \rangle.
\]
Note that $\exp J_z = p$, and therefore 
$\exp \B_0(G_x) \leq p \cdot \exp \B_0(G_x/\langle z \rangle)$. 
Repeating this argument with $G_x$ replaced by $G_x/\langle z \rangle$,
it follows that
\[
\exp \B_0(G_x) \leq |\ker \pi| \cdot \exp \B_0(S_j) \leq p^C \cdot \exp \B_0(S_j).
\]
Groups on the main line have bounded
Bogomolov multipliers by Theorem \ref{th:bog_exp_main_line},
which means that $\exp \B_0(G_x)$ is also bounded. It now follows
from the form \eqref{eq:schur_multiplier_form_in_family}
that $\B_0(G_x)$ eventually stabilizes for large enough $x$.
This completes the proof.
\end{proof}

We provide a non-trivial example obtained using {\sf GAP} \cite{DEF08}.

\begin{example}

Consider the following pro-$3$ group of coclass $3$. Let $P$ be the semidirect
product of $C_9 = \langle x \rangle$ by $C_3 = \langle y \rangle$ with $y$
acting nontrivially on $x$. Let $T = \mathbb{Z}_3^6 = \langle t_1, \dots, t_6
\rangle$ be a free $3$-adic abelian group. Let the generators $x,y$ act on $T$
via the matrices
\[
X = 
\begin{bmatrix}
0 &  2 &  1 &  -1 &  2 &  -2 \\
0 &  1 &  0 &  0 &  1 &  -1 \\
0 &  0 &  0 &  0 &  1 &  -1 \\
0 &  0 &  0 &  0 &  3 &  -2 \\
1 &  1 &  0 &  -1 &  0 &  -1 \\
0 &  3 &  0 &  -1 &  0 &  -1 \\
\end{bmatrix}
\quad
\textrm{and}
\quad
Y = 
\begin{bmatrix}
1 & 0 & 2 & -2 & -2 & 0 \\
0 & 1 & 1 & -1 & -1 & 0 \\
0 & 0 & 1 & -1 & 0 & 0 \\
0 & 0 & 3 & -2 & 0 & 0 \\
0 & 0 & 0 & 0 & -2 & 1 \\
0 & 0 & 0 & 0 & -3 & 1 \\
\end{bmatrix}
\]
The group $S = P \ltimes T$ is a pro-$3$ group of coclass $3$.
It has a finite quotient $G = S / \gamma_5(S)$ of order $3^7$, which is an extension of
$\gamma_3(S)/\gamma_5(S) \cong C_3 \times C_3$ by $P \ltimes T/\gamma_3(S)$
with $T/\gamma_3(S) \cong C_3 \times C_3$. The Bogomolov multiplier of $G$
is nontrivial, in fact $\B_0(G) \cong C_3$.
Hence all the finite quotients $S / \gamma_i(S)$ for $i \geq 5$ have non-trivial
Bogomolov multipliers.
\end{example}


\end{document}